\documentclass[12pt]{article}
\usepackage{amsmath,amsfonts,amssymb,amsthm,verbatim}
\usepackage[hypertex,draft=false]{hyperref}
\newtheorem{theorem}{Theorem}
\newtheorem{proposition}[theorem]{Proposition}
\newtheorem{corollary}[theorem]{Corollary}
\newtheorem{lemma}[theorem]{Lemma}
\theoremstyle{definition}
\newtheorem{definition}[theorem]{Definition}
\newtheorem{example}[theorem]{Example}
\theoremstyle{remark}
\newtheorem{source}[theorem]{SOURCE}

\newcommand{\QQ}{\mathbb{Q}}
\newcommand{\RR}{\mathbb{R}}
\newcommand{\ZZ}{\mathbb{Z}}
\newcommand{\CC}{\mathbb{C}}
\DeclareMathOperator{\codim}{codim}

\title{Strong Lefschetz elements of the coinvariant rings
  of finite Coxeter groups}
\author{Toshiaki Maeno, Yasuhide Numata and Akihito Wachi}

\date{}

\begin{document}
\maketitle

\begin{abstract}
For the coinvariant rings of finite Coxeter groups 
of types other than H$_4$,
we show that a homogeneous element of degree one
is a strong Lefschetz element 
if and only if it is not fixed by any reflections.
We also give the necessary and sufficient condition
for strong Lefschetz elements 
in the invariant subrings of 
the coinvariant rings of Weyl groups.
\end{abstract}



\section{Introduction}
\label{sec:intro}

Let $W$ be a finite Coxeter group generated by reflections
acting on an $n$-dimensional $\RR$-vector space $V$.
The polynomial ring $\RR[V]$ is then equipped 
with the action of $W$.
The {\em coinvariant ring} of $W$ is defined by $R = \RR[V]/J$,
where $J$ is the graded ideal of $\RR[V]$ 
generated by the $W$-invariant polynomials without constant terms.
The first purpose of this paper is 
to give the necessary and sufficient condition
for the {\em strong Lefschetz elements}
(Definition~\ref{defn:slp} below)
of $R$:
\begin{theorem}
\label{thm:condition-for-full-flag}
Let $W$ be a finite Coxeter group 
which does not contain irreducible components of type H$_4$.
Then a homogeneous element $\ell$ of degree one
is a strong Lefschetz element
if and only if $\ell$ is not fixed by any reflections of $W$.
\qed
\end{theorem}

The second purpose is to give the necessary and sufficient condition
for strong Lefschetz elements of the 
{\em parabolic invariant} $R^{W_S}$,
which is by definition the invariant subring of $R$
under the action of a parabolic subgroup $W_S$ of $W$:
\begin{theorem}
\label{thm:condition-for-partial-flag}
Let $W$ be a Weyl group, and $W_S$ a parabolic subgroup of $W$.
Then a homogeneous element $\ell$ of degree one
of the parabolic invariant $R^{W_S}$ 
is a strong Lefschetz element
if and only if $\ell$ is not fixed by any reflections in
$W \smallsetminus W_S$.
\qed
\end{theorem}

The notion of the strong Lefschetz
property for commutative Artinian graded rings is an abstraction of a property
of the cohomology ring $H^*(X, \RR)$ of a compact K\"ahler manifold $X.$
The Hard Lefschetz Theorem
(Proposition~\ref{prop:hard-lefschetz-theorem}. 
  see \cite{MR507725}, e.g.)
tells us that the multiplication by the class of
the K\"ahler form induces an isomorphism 
between $H^i(X, \RR)$ and $H^{2\dim_{\CC}X-i}(X, \RR).$
In view of such a property of the cohomology ring,
the strong Lefschetz property is defined as follows:
\begin{definition}
\label{defn:slp}
A graded ring $R = \bigoplus_{d=0}^m R_d$ 
having a symmetric Hilbert function
is said to have
the {\em strong Lefschetz property},
if there exists an element $\ell \in R_1$ such that
the multiplication map
$\times \ell^{m-2i}: R_i \to R_{m-i}$ ($f \mapsto \ell^{m-2i}f$)
is bijective for every $i = 0,1,\ldots, \lfloor m/2 \rfloor$.
In this case, $\ell$ is called a {\em strong Lefschetz element}.
\end{definition}
\noindent
It should be remarked that
we can define the strong Lefschetz property 
for graded rings with non-symmetric Hilbert functions
(see \cite{MR1970804}, e.g.).
In this paper, we only consider the strong Lefschetz property
for symmetric Hilbert functions.

Our interest is the condition for an element in $R_1$ 
to be a strong Lefschetz element.
 It is not a trivial problem even in the case of
the cohomology ring $H^*(X,\RR)$ of a K\"ahler manifold $X.$
In fact, the K\"ahler cone of $X$ is only
part of the set of the strong Lefschetz elements in $H^*(X,\RR).$

When $W$ is a Weyl group
(i.e. crystallographic Coxeter group),
its coinvariant ring $R$ is isomorphic to 
the cohomology ring of the flag variety $G/B$,
where $G$ is the Lie group corresponding to $W$,
and $B$ is a Borel subgroup of $G$.
Moreover the parabolic invariant $R^{W_S}$ is isomorphic to
the cohomology ring of the partial flag variety $G/P$,
where $P$ is the parabolic subgroup of $G$
whose Weyl group is $W_S$.
Thus $R$ and $R^{W_S}$ have the strong Lefschetz property
for Weyl groups $W$ from the geometric result.
By an additional algebraic argument,
Theorem~\ref{thm:condition-for-full-flag}
(resp. Theorem~\ref{thm:condition-for-partial-flag})
determines the whole set of the strong Lefschetz elements of $R$
(resp. $R^{W_S}$)
for Weyl groups $W$,
and the set of the strong Lefschetz elements
includes the K\"ahler cone
as one of the connected components.

When $W$ is not a Weyl group
(i.e. non-crystallographic Coxeter group),
we have to study the strong Lefschetz property
of the coinvariant ring $R$ by algebraic arguments,
since we do not have geometric realizations 
of the coinvariant rings.
For the non-crystallographic part 
of Theorem~\ref{thm:condition-for-full-flag},
we give the necessary and sufficient condition
for strong Lefschetz elements of $R$
for Coxeter groups of type H$_3$ or I$_2(m)$. 
In particular, Theorem~\ref{thm:condition-for-full-flag}
proves the strong Lefschetz property of the coinvariant rings
of types H$_3$ and I$_2(m)$.
For the type H$_4$,
the strong Lefschetz property of the coinvariant ring
is proved by \cite{MR2371985},
but the necessary and sufficient condition 
for strong Lefschetz elements is not given yet.

This paper is organized as follows:
In Section~\ref{sec:coinvariant-ring}
we review the coinvariant rings of finite Coxeter groups,
and we give the necessary condition for the strong Lefschetz elements
in Section~\ref{sec:necessary-condition}.
In Section~\ref{sec:weyl-group}, 
we review the coinvariant rings of Weyl groups
from a geometric viewpoint.
The necessary condition given 
in Section~\ref{sec:necessary-condition}
is proved to be also sufficient 
in Section~\ref{sec:weyl-group} for Weyl groups
and in Section~\ref{sec:sufficient-condition} for Coxeter groups
of types H$_3$ and I$_2(m)$.
The proof of Theorem~\ref{thm:condition-for-full-flag} is
completed here.
In Section~\ref{sec:parab-invariant} we study
the strong Lefschetz elements of parabolic invariants
of coinvariant rings of Coxeter groups,
and prove Theorem~\ref{thm:condition-for-partial-flag}.

We would like to thank Hideaki Morita who posed us 
a problem on the main theorems.
We would also like to thank Junzo Watanabe for his valuable comments 
on this study.
The first author is partially supported 
by JSPS Grant-in-Aid for Scientific Research \#{19740013}.

\section{The coinvariant rings of finite Coxeter groups}
\label{sec:coinvariant-ring}

In this section 
we review the structure of coinvariant rings of finite
Coxeter groups briefly.
Irreducible finite Coxeter groups are classified as
Table~\ref{table:coxeter-group}
(see \cite{MR1066460} for instance).
Irreducible Coxeter groups of types from A to G are
said to be {\em crystallographic},
and the others are said to be {\em non-crystallographic}.
Crystallographic ones are also called {\em Weyl groups}.
Finite Coxeter groups decompose into direct products
of irreducible Coxeter groups.

Let $W$ be a finite Coxeter group
generated by $n$ simple reflections $s_1, s_2, \ldots, s_n$
acting on an $n$-dimensional $\RR$-vector space $V$.
Let $l(w)$ be the {\em length} of $w \in W$,
which is by definition the length $l$ of the shortest expression
$w = s_{i_1} s_{i_2} \cdots  s_{i_l}$ for $w$.
It is known that there exists a unique element $w_0 \in W$
of maximum length,
which is called the {\em longest element},
and that the length of $w_0$ is equal to the number of reflections
in $W$
(see \cite{MR1066460}, e.g.).
\begin{table}
  \caption{finite Coxeter groups}
  \label{table:coxeter-group}
  $
  \begin{array}{|l|cccccc|}
    \hline
    \text{type} &
    \mathrm{A}_n & \mathrm{B}_n & \mathrm{D}_n & \mathrm{E}_6 &
    \mathrm{E}_7 & \mathrm{E}_8
    \\ \hline
    \rule{0pt}{2.3ex}
    \# W &
    (n+1)! & 2^n n! & 2^{n-1} n! & 2^7 3^4 5 &
    2^{10} 3^4 5~7 & 2^{14} 3^5 5^2 7
    \\ \hline
    \rule{0pt}{2.3ex}
    \#\text{(reflections)} &
    n(n+1)/2 & n^2 & n(n-1) & 36 &
    63 & 120
    \\
    \hline
  \end{array}
  $
  \smallskip \\
  $
  \begin{array}{|l|cc||ccc|}
    \hline
    \text{type} &
    \mathrm{F}_4 & \mathrm{G}_2 & \mathrm{H}_3 & \mathrm{H}_4 &
    \mathrm{I}_2(m)
    \\ \hline
    \rule{0pt}{2.3ex}
    \# W &
    2^7 3^2 & 12 & 120 & 2^6 3^2 5^2 & 2m
    \\ \hline
    \#\text{(reflections)} &
    24 & 6 & 15 & 60 & m
    \\
    \hline
  \end{array}
  $
\end{table}
Let $J$ be the graded ideal of the polynomial ring $\RR[V]$
generated by the $W$-invariant polynomials without constant terms,
and $R = \RR[V]/J$ the {\em coinvariant ring} of $W$.
Then it is known that the dimension of the vector subspace $R_d$
of homogeneous degree $d$ is equal to
the number of elements in $W$ with length $d$,
and hence one has the homogeneous decomposition of $R$
\cite{MR0429933,MR630960}:
\[
  R = \bigoplus_{d=0}^m R_d
  \qquad
  (m = l(w_0), \quad \dim_{\RR} R_d = \#\{w \in W \;;\; l(w)=d\} ).
\]
The Hilbert function of $R$ is {\em symmetric},
that is,
$\dim_{\RR} R_d = \dim_{\RR} R_{m-d}$ 
for $d = 0, 1, \ldots, \lfloor m/2 \rfloor$.

Let $W$ be a Weyl group, and $\Delta$ its root system.
Let $G$ be a connected, simply-connected and semi-simple
complex Lie group corresponding to $\Delta$,
and $B$ its Borel subgroup.
On the one hand,
we have the cohomology ring
$H^*(G/B, \RR)$ of the flag variety $G/B$ .
It follows from the Hard Lefschetz theorem
(Proposition~\ref{prop:hard-lefschetz-theorem})
that $H^*(G/B, \RR)$ has the strong Lefschetz property.
On the other hand,
we have the coinvariant ring
$R = \RR[\mathfrak{h}^*]/J$ of $W$,
where $\mathfrak{h}^*$ is the real vector space 
spanned by the root system $\Delta$,
and therefore $W$ acts on $\mathfrak{h}^*$.
It is known that $\mathfrak{h}^*$ is canonically 
isomorphic to $H^2(G/B, \RR)$,
and
\begin{align}
  \label{eq:coinvariant-ring-isom-to-cohomology-ring}
  R \simeq H^*(G/B, \RR)
  \qquad
  \text{with~}
  R_d \simeq H^{2d}(G/B, \RR),
\end{align}
(see \cite{MR0051508},
  \cite[1.3.~Proposition]{MR0429933},
  \cite[Theorem~2.7]{MR630960}, e.g.).
In particular,
it turns out that the coinvariant ring $R$ 
has the strong Lefschetz property.

When $W$ is a non-crystallographic Coxeter group,
there is no geometric realization of the coinvariant ring.
However, the coinvariant rings have the strong Lefschetz property
by Subsections~\ref{subsec:I2(m)} and \ref{subsec:H3}
for H$_3$ and I$_2(m)$
and by \cite{MR2371985} for H$_4$.
Summarizing the above arguments,
we have the following:
\begin{proposition}
\label{prop:coinvariant-ring-have-SLP}
The coinvariant rings of finite Coxeter groups have 
the strong Lefschetz property.
\end{proposition}

\section{The necessary condition for strong Lefschetz elements}
\label{sec:necessary-condition}

As noted in Proposition~\ref{prop:coinvariant-ring-have-SLP},
the coinvariant rings $R$ of finite Coxeter groups $W$
have the strong Lefschetz property.
In this section 
we give the necessary condition for the strong Lefschetz elements
of $R$.
\begin{theorem}
\label{thm:necessary-condition}
Let $R = \bigoplus_{d=0}^m R_d$ be the coinvariant ring
of a finite Coxeter group $W$
(possibly non-crystallographic).
If $\ell \in R_1$ is a strong Lefschetz element,
then $\ell$ is not fixed by any reflections of $W$.
\end{theorem}
\begin{proof}
By \cite{MR0429933}, \cite{MR630960},
the subspace $R_m$ of the maximum degree is one-dimensional,
and spanned by an anti-invariant element,
where an element $f \in R$ is said to be anti-invariant
if $f$ is sent to $-f$ by any reflections of $W$.

Take a strong Lefschetz element $\ell \in R_1$,
and assume that $\ell$ is fixed by a reflection $s \in W$.
On the one hand $\ell^m = 0$,
since $\ell^m \in R_m$ is also fixed by $s$,
and each element in $R_m$ is anti-invariant.
On the other hand $\ell^m \ne 0$,
since the multiplication map $\times \ell^m: R_0 \to R_m$
is bijective by the definition of strong Lefschetz elements.
This is a contradiction.
We thus have proved that
a strong Lefschetz element is not fixed by any reflections.
\end{proof}

\begin{example}[type A$_{n-1}$]
\label{ex:An-1}
Let $W$ be the symmetric group $S_n$ on $n$ letters,
and acting on the polynomial ring $\RR[x_1, x_2, \ldots, x_n]$
via permutation of the variables.
Let $R$ be the coinvariant ring $\RR[x_1,x_2,\ldots,x_n]/J$,
where $J$ is the graded ideal generated
by the symmetric polynomials without constant terms.
In this case, by Theorem~\ref{thm:necessary-condition},
if $a_1 x_1 + a_2 x_2 + \cdots + a_n x_n \in R_1$ is a
strong Lefschetz element,
then $a_i \ne a_j$ for any $i \ne j$.
\qed
\end{example}

\section{Coinvariant rings of Weyl groups as the cohomology rings of flag
varieties}
\label{sec:weyl-group}

In this section, we study the strong Lefschetz elements
of the coinvariant rings of the crystallographic Coxeter groups
(i.e. Weyl groups).
If the
Coxeter group $W$ is crystallographic, then its coinvariant ring is
isomorphic to the cohomology ring of the corresponding flag variety
(Equation~(\ref{eq:coinvariant-ring-isom-to-cohomology-ring})).
The argument in this section is based on the geometric property of the
flag variety, so it is not applicable to the case of non-crystallographic
Coxeter groups. Here let us remind of the Hard Lefschetz Theorem
(see \cite{MR507725}, e.g.).
\begin{proposition}[The Hard Lefschetz Theorem]
\label{prop:hard-lefschetz-theorem}
Let $(X,\omega)$ be a compact K\"ahler manifold.
Denote by $L$ the multiplication operator by the K\"ahler class $[\omega] \in
H^2(X,\RR).$ Then
\[ L^{\dim_{\CC}X-i}:H^i(X,\RR) \rightarrow H^{2\dim_{\CC}X-i}(X,\RR) \]
is an isomorphism for $i=0,\ldots , \dim_{\CC}X.$
\end{proposition}

Next, let us recall Kleiman's criterion for ampleness. For a projective
complex manifold $X,$ we define
\[ N^1(X)={\rm Pic}(X)_{\RR}/(\textrm{numerical equivalence}) , \]
where ${\rm Pic}(X)_{\RR}$ is the Picard group of $X$.
Let $Z_1(X)_{\RR}$ be the $\RR$-linear space of the formal linear combinations
of curves on $X.$ We also define
\[ N_1(X)=Z_1(X)_{\RR}/(\textrm{numerical equivalence}) . \]
Denote by $\overline{{\rm NE}}(X)$ the closed convex cone 
in $N_1(X)$ spanned by the images of curves in $Z_1(X).$
\begin{proposition}[Kleiman's criterion {\cite{MR0206009}}]
An element $L\in N^1(X)$ is contained in the ample cone if and only if
$(L,C)>0$ for all $C\in \overline{{\rm NE}}(X)\setminus \{ 0 \}.$
\end{proposition}

Now we consider the flag variety $G/B$
corresponding to a Weyl group $W$
as in Section~\ref{sec:coinvariant-ring}.
Let $\Delta$ be the root system corresponding to $W$,
and $\Sigma$ the set of the simple roots of $\Delta$
determined by the Borel subgroup $B$.
Let $\mathfrak{h}^*$ be the real vector space spanned by $\Delta$, 
which is canonically isomorphic to $H^2(G/B, \RR)$,
and $\mathfrak{h}$ the dual space of $\mathfrak{h}^*$,
which is canonically isomorphic to $H_2(G/B, \RR)$.
More generally,
if $P$ is a parabolic subgroup of $G$ and $\Sigma_P \subset \Sigma$ 
is the corresponding subset, then we have 
\[ H^2(G/P, \mathbb{R})=\mathfrak{h}_P^*:=\bigoplus_{\alpha \not\in \Sigma_P}
\mathbb{R} \cdot \alpha \subset \mathfrak{h}^*, \]
\[ H_2(G/P, \mathbb{R})=\mathfrak{h}_P:=\mathfrak{h}/
\bigoplus_{\alpha^{\vee} \in \Sigma_P^{\vee}}
\mathbb{R} \cdot \alpha^{\vee}. \]
We regard the set $\Sigma^\vee$ of the dual of the simple roots
as a subset of $\mathfrak{h}$,
where the dual root $\alpha^\vee \in \mathfrak{h}$
of a root $\alpha \in \Delta$
is defined by
$\langle \alpha', \alpha^\vee \rangle =
2(\alpha',\alpha)/(\alpha,\alpha)$
(for any $\alpha' \in \Delta$)
using the natural pairing $\langle\,{,}\,\rangle$
of $\mathfrak{h}^*$ with $\mathfrak{h}$
and the invariant inner product $(,)$ on $\mathfrak{h}^*$.
Then
\[ \overline{{\rm NE}}(G/P)=\sum_{\alpha^{\vee}\in \Sigma \setminus \Sigma_P}
\mathbb{R}_{\geq 0}\cdot \alpha^{\vee} \] 
under the identification $H_2(G/P,\mathbb{R})=\mathfrak{h}_P.$ 
In fact, Chevalley \cite{MR0106966} has shown 
that the line bundle $L_{\lambda}$ over $G/B$
corresponding to a weight $\lambda \in \mathfrak{h}$ is ample if and only if
$\lambda$ satisfies $\langle \lambda , \alpha^{\vee} \rangle >0,$ for all 
$\alpha \in \Sigma^{\vee}$. 
Since every ample line bundle corresponds to the class of a K\"ahler form in
$H^2(G/P,\RR)$ and hence it gives a strong Lefschetz element in
$H^*(G/P,\RR)$ from the Hard Lefschetz Theorem, we have the following.
\begin{proposition}
If $\ell\in \mathfrak{h}_P^*=H^2(G/P, \RR)$ satisfies the condition
$\langle \ell, \alpha^{\vee} \rangle > 0$ 
for all $\alpha^{\vee}\in \Sigma^{\vee} \smallsetminus \Sigma_P^{\vee},$
then $\ell$ is a strong Lefschetz element in 
$H^*(G/P,\mathbb{R})$.
\end{proposition}
The Weyl chambers with respect to the root system $\Delta$ is 
by definition the connected components of 
$\{ x \in \mathfrak{h}^* \;;\; \langle x, \alpha^\vee \rangle \ne 0
\text{~for all~} \alpha \in \Delta \}$,
and the fundamental Weyl chamber is the connected component
$\{ x \in \mathfrak{h}^* \;;\; \langle x, \alpha^\vee \rangle > 0 
\text{~for all~} \alpha^\vee \in \Sigma^\vee
\}$,
which is equal to the K\"ahler cone 
under the identification $H^2(G/B,\RR)=\mathfrak{h}^*.$
Note that the image of a strong Lefschetz element of $R$
under an action of the Weyl group is also a strong Lefschetz element.
Since the Weyl group acts on the set of the Weyl chambers transitively,
we obtain the following sufficient condition for
 the strong Lefschetz element.
\begin{corollary}
\label{cor:sufficient-condition-for-weyl-group}
Let $R$ be the coinvariant ring of the Weyl group.
An element $\ell\in R_1$ is a strong Lefschetz element 
if $\ell$ is not fixed by any reflections of $W$.
\end{corollary}

\section{The sufficient condition for strong Lefschetz elements}
\label{sec:sufficient-condition}

In this section 
we prove that the condition in Theorem~\ref{thm:necessary-condition}
is also the sufficient condition for the strong Lefschetz elements
of the coinvariant rings $R$ of Coxeter groups $W$
of types other than H$_4$.
Thus our first main theorem (Theorem~\ref{thm:condition-for-full-flag})
is proved by
Theorems~\ref{thm:necessary-condition} and
\ref{thm:sufficient-condition-Coxeter-group} below.
\begin{theorem}
\label{thm:sufficient-condition-Coxeter-group}
Let $W$ be a finite Coxeter group 
which does not contain irreducible components of type H$_4$.
Then a homogeneous element $\ell$ of $R_1$ of degree one
is a strong Lefschetz element of $R$
if $\ell$ is not fixed by any reflections of $W$.
\qed
\end{theorem}
If a Coxeter group $W$ is a direct product of 
two Coxeter groups, 
then the coinvariant ring of $W$ is the tensor product 
of their coinvariant rings.
Thus we may assume that $W$ is an irreducible Coxeter group.
We already give the proof of 
Theorem~\ref{thm:sufficient-condition-Coxeter-group}
for Weyl groups
(Corollary~\ref{cor:sufficient-condition-for-weyl-group}),
and the remaining non-crystallographic Coxeter groups are of types
H$_3$, H$_4$ and I$_2(m)$.
We give the proof of 
Theorem~\ref{thm:sufficient-condition-Coxeter-group}
for the types H$_3$ and I$_2(m)$
in the following subsections.
Theorem~\ref{thm:sufficient-condition-Coxeter-group}
is still conjectural for H$_4$,
although the strong Lefschetz property of the coinvariant ring 
of type H$_4$ has been proved by \cite{MR2371985}.

\subsection{The coinvariant ring of I$_2(m)$}
\label{subsec:I2(m)}

In this subsection 
we prove Theorem~\ref{thm:sufficient-condition-Coxeter-group}
for the irreducible Coxeter group of type I$_2(m)$.
As a group acting on $\RR^2$,
the dihedral group I$_2(m)$ is generated by two simple reflections
$s_1 = s_{\beta(0)}$ and
$s_2 = s_{\beta((m-1)\pi/m)}$,
where $\beta(\psi) \in \RR^2$ is the unit vector of angle $\psi$,
and $s_{\beta(\psi)} \in GL_2(\RR)$ denotes the reflection
with respect to the vector $\beta(\psi)$.
These generators are subject to the relation $(s_1 s_2)^m = e$,
where $e$ denotes the identity element of I$_2(m)$.

Put $\theta = \pi/m$.
The group I$_2(m)$ has $m$ reflections
$s_{\beta(k\theta)}$
corresponding to the positive roots $\beta(k\theta)$
($0 \le k \le m-1$) 
and $2m$ elements as in the following lemma:

\begin{lemma}
Let $\alpha_1 = \beta(0)$ and $\alpha_2 = \beta((m-1)\theta)$
be the simple roots.
Define elements $a_k$ and $b_k$ by
alternating products of $s_1$ and $s_2$ as
\begin{align*}
  a_k &=   s_1 s_2 s_1 s_2 \cdots s_i
  \quad (\text{product of $k$ reflections beginning with $s_1$.
      $i=1$ or $2$}), 
  \\
  b_k &= s_2 s_1 s_2 s_1 \cdots s_j
  \quad (\text{product of $k$ reflections beginning with $s_2$.
      $j=1$ or $2$}),
\end{align*}
for $k \ge 0$.
Note that $a_k$ and $b_k$ are of length $k$, only when $0\le k \le m$.

Then the diagram of the Bruhat order on I$_2(m)$ is as follows:
\begin{gather}
  \label{eq:bruhat-order-of-I2(m)}
  \begin{array}{cccccccccccccccccc}
    && a_1 &
    \xrightarrow{\beta(\theta)}
    & a_2 & \cdots & a_k &
    \xrightarrow{\beta(k\theta)}
    & a_{k+1} & \cdots & a_{m-1}
    \\
    e &
    \makebox[0pt][c]{\raisebox{2.1ex}{{\scriptsize$\alpha_1$}}}
    \makebox[0pt][l]{\raisebox{1.1ex}{$\nearrow$}}
   \makebox[0pt][c]{\raisebox{-1.6ex}{{\scriptsize$\alpha_2$}}}
    \makebox[0pt][l]{\raisebox{-1.1ex}{$\searrow$}}
    &&
    \makebox[0pt][l]{\raisebox{2.1ex}{\scriptsize$\alpha_2$}$\!\!\searrow$}
    \makebox[18pt][l]{\raisebox{-1.6ex}{\scriptsize$\alpha_1$}$\!\!\nearrow$}
    &&&&
    \makebox[0pt][l]{\raisebox{2.1ex}{\scriptsize$\alpha_2$}$\!\!\searrow$}
    \makebox[18pt][l]{\raisebox{-1.6ex}{\scriptsize$\alpha_1$}$\!\!\nearrow$}
    &&&&
    \makebox[0pt][l]{\raisebox{2.1ex}{{\ \ \scriptsize$\alpha_2$}}}
    \makebox[0pt][l]{\raisebox{1.1ex}{$\searrow$}}
    \makebox[0pt][l]{\raisebox{-1.6ex}{{\ \ \scriptsize$\alpha_1$}}}
    \makebox[8pt][l]{\raisebox{-1.1ex}{$\nearrow$}}
    &
    a_m = b_m,
    \\
    && b_1 &
    \xrightarrow[\beta((m-2)\theta)]{}
    & b_2 & \cdots & b_k &
    \xrightarrow[\beta((m-k-1)\theta)]{}
    & b_{k+1} & \cdots & b_{m-1}
  \end{array}
\end{gather}
where $w \overset{\beta}{\to} w'$ means that
$s_\beta w = w'$,
$l(w) + 1 = l(w')$ and
$\beta$ is a positive root.

In particular,
the Hilbert function $H(d) = \dim_{\RR} R_d$ of the coinvariant ring 
$R = \bigoplus_{d=0}^m R_d$ is given by
$(H(d))_{d=0}^{m} = (1,2,2,\ldots, 2,1)$.
\end{lemma}
\begin{proof}
It is clear that 
there are two elements of length $k$ in I$_2(m)$,
and they are $a_k$ and $b_k$ ($1 \le k \le m-1$).
Between the elements of length $k$ and those of length $k+1$,
it is easy to show the relations
$s_2 a_k = b_{k+1}$,
$s_1 b_k = a_{k+1}$,
$s_{\beta(k\theta)} a_k = a_{k+1}$
and $s_{\beta((m-k-1)\theta)} b_k = b_{k+1}$.
Therefore we have the desired diagram.
\end{proof}

Let $\ell \in R$ be an element 
not fixed by any reflections of I$_2(m)$.
We divide the proof of 
Theorem~\ref{thm:sufficient-condition-Coxeter-group}
for I$_2(m)$ into two steps.

\paragraph{Step 1: $\times \ell^{m}: R_0 \to R_m$ is bijective:}
Note first that there exists an element $g \in R_1$ with $g^m \ne 0$.
Indeed, we have
\begin{align*}
(-1)^m m! u_1 u_2 \cdots u_m = 
\sum_{I \subset \{1,2,\ldots,m\}}
(-1)^{\#I} \left( \sum_{i\in I} u_i \right)^m,
\end{align*}
in the polynomial ring $\ZZ[u_1,u_2,\ldots,u_m]$
(see \cite[Lemma~2.4.1]{MR1328644}, e.g.).
We substitute the $m$ positive roots to $u_1, u_2, \ldots, u_m$,
then the left-hand side is nonzero in $R_m$
(\cite[Corollary~3.3]{MR630960}).
Hence there exists $I \subset \{1,2,\ldots,m\}$
such that $\left( \sum_{i\in I} u_i \right)^m \ne 0$,
and $g =  \sum_{i\in I} u_i \in R_1$ is the desired element.

Let $\ell = a X_1 + b X_2$ be an element
of $R_1$ not fixed by any reflections ($a,b \in \RR$).
The (unique) matrix coefficient of 
the map $\times \ell^m$ is a
polynomial $f$ in $a$ and $b$ of total degree $m$,
and not identically zero as noted above.
The polynomial $f$ is divisible by 
every linear form defining a reflecting hyperplane,
since it becomes zero when $\ell$ is fixed by a reflection
(see Theorem~\ref{thm:necessary-condition}).
The number of the reflections is equal to $m$,
and therefore the polynomial $f$ should be the product of
$m$ linear forms defining the reflecting hyperplanes.
Hence $\times \ell^m: R_0 \to R_m$ is bijective,
since $\ell$ is not fixed by any reflections,
that is, not on any reflecting hyperplanes.

\paragraph{Step 2: $\times \ell: R_k \to R_{k+1}$ is bijective
($1 \le k \le m-2$):}
If we prove this step,
then 
$\times \ell^{m-2i}: R_i \to R_{m-i}$ 
is bijective for every $i = 1,\ldots, \lfloor m/2 \rfloor$.
Hence, together with Step 1,
the element $\ell$ is proved to be a strong Lefschetz element.
In the rest of this subsection we prove this step.

By \cite{MR630960}, 
the coinvariant ring $R$ of a Coxeter group $W$
has a linear basis $\{ X_w \;|\; w \in W \}$,
where $X_w$ is a homogeneous element of degree $l(w)$,
and they enjoy the following {\em Pieri formula}
\cite[Corollary~4.6]{MR630960}:
\begin{align}
  \label{eq:pieri-formula}
  X_{s_\alpha} X_w &=
  \sum_{\substack{\text{$\beta$: positive root},\\ l(w s_\beta) = l(w)+1}}
  \frac{2(\beta, \varpi_\alpha)}{(\beta, \beta)}
  X_{w s_\beta}
  \qquad
  \text{($\alpha$: simple root)},
\end{align}
where $\varpi_\alpha \in R_1$ denotes the fundamental weight
corresponding to a simple root $\alpha$
(i.e. $2(\alpha', \varpi_{\alpha})/(\alpha',\alpha') 
  = \delta_{\alpha,\alpha'}$
for every simple root $\alpha'$).
In the case of $W = \text{I}_2(m)$,
the simple roots are $\beta(0)$ and $\beta((m-1)\theta)$,
and it is easy to see that 
the corresponding fundamental weights are
$\beta( \pi/2 - \theta ) / 2\sin\theta$ and
$\beta(0) / 2\sin\theta$, respectively.
Thus using the Pieri formula~(\ref{eq:pieri-formula}),
we can compute as 
\begin{align*}
  X_{s_1} X_{a_k^{-1}} &=
  \sum_{\beta = \beta((m-1)\theta), \beta(k\theta)}
  \frac{2(\beta, \varpi_1)}{(\beta, \beta)}
  X_{a_k^{-1} s_\beta}
  \\ &=
  2 \left(
    \beta((m-1)\theta), \frac{\beta(\pi/2-\theta)}{2\sin\theta}
  \right)
  X_{b_{k+1}^{-1}} +
  2 \left(
    \beta(k\theta),  \frac{\beta(\pi/2-\theta)}{2\sin\theta}
  \right)
  X_{a_{k+1}^{-1}}
  \\ &=
  \frac{\sin(k+1)\theta}{\sin\theta} X_{a_{k+1}^{-1}}.
\end{align*}
One can similarly compute the other actions of $\times X_{s_i}$
on $X_{a_k^{-1}}$ and on $X_{b_k^{-1}}$:
\begin{align*}
  X_{s_1} X_{b_k^{-1}} &=
  X_{a_{k+1}^{-1}} + \frac{\sin k\theta}{\sin\theta} X_{b_{k+1}^{-1}},
  \\
  X_{s_2} X_{a_k^{-1}} &=
  X_{b_{k+1}^{-1}} + \frac{\sin k\theta}{\sin\theta} X_{a_{k+1}^{-1}},
  &
  X_{s_2} X_{b_k^{-1}} &=
\frac{\sin(k+1)\theta}{\sin\theta} X_{b_{k+1}^{-1}}.
\end{align*}
In other words, the matrix representations of
the multiplication maps $\times X_{s_i}: R_k \to R_{k+1}$
with respect to the bases $\{ X_{a_k^{-1}}, X_{b_k^{-1}} \}$
and  $\{ X_{a_{k+1}^{-1}}, X_{b_{k+1}^{-1}} \}$ are
\begin{align}
  \label{eq:matrix-of-X-for-I2(m)}
  X_{s_1} &= \begin{pmatrix} p_k & 1 \\ 0 & p_{k+1} \end{pmatrix}, &
  X_{s_2} &= \begin{pmatrix} p_{k+1} & 0 \\ 1 & p_k \end{pmatrix},
  \qquad \text{where~} p_k = \frac{\sin k\theta}{\sin\theta}.
\end{align}
If we write $\ell = a X_{s_1} + b X_{s_2}$ ($a, b \in \RR$),
then the matrix representation of 
$\times \ell: R_k \to R_{k+1}$ is
\begin{align*}
  \begin{pmatrix} a p_k + b p_{k+1}  & a \\
    b & a p_{k+1} + b p_k
  \end{pmatrix},
\end{align*}
by Equation (\ref{eq:matrix-of-X-for-I2(m)}),
and its determinant is equal to
$(a p_k + b p_{k+1})(a p_{k+1} + b p_k) - ab
= (a^2+b^2)p_k p_{k+1} + ab(p_k^2 + p_{k+1}^2 - 1)$.
Here $a$ or $b$ is nonzero, 
since $\ell$ is not fixed by any reflections,
and we may assume that $b \ne 0$ without loss of generality.
Set $t = a/b$, and the determinant then equals
$b^2(p_k p_{k+1} t^2 + (p_k^2 + p_{k+1}^2 - 1)t + p_k p_{k+1})$.
To show that this expression has no zeros
it suffices to show that the discriminant of this quadratic
polynomial in $t$ is negative.
The discriminant is
$(p_k^2 + p_{k+1}^2 - 1)^2 - 4(p_k p_{k+1})^2$,
and it is equal to
\begin{align*}
  & (\sin k\theta + \sin(k+1)\theta + \sin\theta)
  (\sin k\theta + \sin(k+1)\theta - \sin\theta)
  \\ & \qquad\times
  (\sin k\theta - \sin(k+1)\theta + \sin\theta)
  (\sin k\theta - \sin(k+1)\theta - \sin\theta) / \sin^4 \theta.
\end{align*}
It is easy to see that the fourth factor is negative,
and the other factors are positive.
We thus have proved that $\times \ell: R_k \to R_{k+1}$
is bijective for $k=1,2,\ldots,m-2$,
and hence Step 2 is completed.

\subsection{The coinvariant ring of H$_3$}
\label{subsec:H3}

In this subsection
we prove Theorem~\ref{thm:sufficient-condition-Coxeter-group}
for the type H$_3$
using the computer algebra system Macaulay 2.
The Coxeter group of type H$_3$ acts on $\RR^3$
and hence on $\RR[x_1, x_2, x_3]$.
By \cite[2.6]{MR926338},
the fundamental invariant polynomials of H$_3$ are
$I_{2\cdot 1}$, $I_{2\cdot 3}$ and $I_{2\cdot 5} \in \RR[x_1,x_2,x_3]$
given by
\[
  I_{2\cdot k}=
  \sum_{(i,j)=(1,2),(2,3),(3,1)}
  (\tau x_i + x_j)^{2k}+(\tau x_i - x_j)^{2k}, 
  \quad
  \tau=\frac{\sqrt{5}+1}{2}.
\]
Thus the coinvariant ring $R$ is realized as
$\RR[x_1,x_2,x_3]/(I_{2\cdot 1},I_{2\cdot 3},I_{2\cdot 5})$.
The socle degree of $R$ is equal to $15$.

We can take the simple roots 
$\alpha_1 = x_1 - \tau x_2 - \tau^2 x_3$,
$\alpha_2 = x_2$
and $\alpha_3 = x_3$
\cite[2.6]{MR926338},
and the corresponding fundamental weights are
$\varpi_1 = x_1$,
$\varpi_2 = \tau x_1 + x_2$,
and $\varpi_3 = \tau^2 x_1 + x_3$
up to positive multiples.
We take $\ell \in R_1$ as 
\begin{align*}
  \ell = a_1 \varpi_1 + a_2 \varpi_2 + a_3 \varpi_3 =
 (a_1 + \tau a_2 + \tau^2a_3)x_1 + a_2x_2 + a_3x_3,
\end{align*}
where $a_1$, $a_2$ and $a_3$ are any positive numbers.
Then $\ell$ is in the fundamental Weyl chamber.
By the transitivity of the group action on the Weyl chambers,
it suffices to show that $\ell$ is a strong Lefschetz element
to prove Theorem~\ref{thm:sufficient-condition-Coxeter-group}.
We compute the determinants of the linear maps
$\times \ell^{15-2i}:R_i \to R_{15-i}$ ($i=0,1,\ldots,7$)
using the following program of Macaulay 2.

\begin{source}\label{getDet}~\\
{\footnotesize
\begin{verbatim}
ACoef = QQ[tau];
ICoef = ideal(tau*tau-tau-1);
RCoef = ACoef / ICoef;
toField RCoef;

I2 = (k, X1, X2, X3) -> ((tau*x1 + x2)^(2*k) + (tau*x1 - x2)^(2*k) +
                         (tau*x2 + x3)^(2*k) + (tau*x2 - x3)^(2*k) +
                         (tau*x3 + x1)^(2*k) + (tau*x3 - x1)^(2*k));
IdealOfInvariants = (X1, X2, X3) -> ideal({ I2( 1, X1, X2, X3),
                                            I2( 3, X1, X2, X3),
                                            I2( 5, X1, X2, X3) });

A' = RCoef[x1, x2, x3];
R' = A' / IdealOfInvariants(x1, x2, x3);

matrixOfBasisOfR' = apply( { 0, 1, 2, 3, 4, 5, 6, 7, 8, 9, 10,
                  11, 12, 13, 14, 15 } , k -> basis( {k}, R' ) );
BasisOfR' = apply( matrixOfBasisOfR', m -> first entries(m));

A = RCoef[x1, x2, x3, a1, a2, a3, b1, b2, b3];
R = A / IdealOfInvariants(x1, x2, x3);

R'2R = map( R , R' );
BasisInR = apply( BasisOfR', 
     basisOfHomogeSp -> apply(  basisOfHomogeSp , f -> R'2R(f)));

l = b1*x1 + b2*x2 + b3*x3;

coef = f ->  coefficients({0, 1, 2}, f);
vectorPresentation = ( tworowedarray, indexset ) -> (
 h = hashTable transpose apply(tworowedarray, m -> first entries m);
 apply(indexset, key -> (if h#?key then h#key else 0)) );

scan({0, 1, 2, 3, 4, 5, 6, 7}, i ->( 
d1=i; d2=15-i;
<< "[[ l^"<< r <<": R" << d1 << "->R" << d2 << " ]]" << endl;
 << "  [ basis ]" << endl;
basisOfDomain = BasisInR_(i);
basisOfRange  = BasisInR_(15-i);
 << "  [ image of basis of domain ]" << endl;
imageOfBasis = apply( basisOfDomain, m -> ( m*l^(15-2*i) ) );
 << "  [ presentation matrix ]" << endl;
presentMatrix = matrix apply( imageOfBasis, polynomialOfImage -> 
      vectorPresentation( coef(polynomialOfImage), basisOfRange ));
 << "  [ det of matrix ]" << endl;
detOfl = det presentMatrix;
 << "  [ subsut det of matrix ]" << endl;
detOflInFWC = substitute( detOfl, 
      { b1 => (a1 + tau*a2 + tau^2*a3), b2 => a2, b3 => a3 });
 << "  [ print substed det ]" << endl;
print( toString detOflInFWC );
))

\end{verbatim}
}
\end{source}

From the output of Source \ref{getDet}, 
we get the determinants $f_i$ of
$\times \lambda= \ell^{15-2i}:R_{i} \to R_{15-i}$.
The determinants $f_i$ are polynomials in $a_1, a_2, a_3$,
in which monomials are of the form 
$(\alpha\tau + \beta) a_1^s a_2^t a_3^u$ 
($\alpha, \beta \in \QQ$).
As listed in Table~\ref{tableofcoef},
the coefficients $\alpha\tau + \beta$ are
all positive for $i=0,2,4,6,7$,
and all negative for $i=1,3,5$.
Therefore the determinants $f_i$ are nonzero for every $i$,
since $a_1,a_2,a_3>0$.
Thus we have proved that $\ell$ is a Lefschetz element,
and hence
Theorem~\ref{thm:sufficient-condition-Coxeter-group}
for H$_3$.

\begin{table}
\caption{The number of monomials in $f_i$ with coefficient
$\alpha \tau+\beta$.}
\medskip
\label{tableofcoef}
\hfil
\begin{tabular}{|c||c|c|c|c|c|c||c||r|}
\hline
$f_i$&
$\substack{\alpha>0,\\ \beta>0}$&
$\substack{\alpha=0,\\ \beta>0}$&
$\substack{\alpha>0,\\ \beta=0}$&
$\substack{\alpha<0,\\ \beta<0}$&
$\substack{\alpha=0,\\ \beta<0}$&
$\substack{\alpha<0,\\ \beta=0}$&
total & \\\hline\hline
$f_0$ & 82   & 1 & 1 & 0    & 0 & 0 & 84   & 1.3      sec\\\hline
$f_1$ & 0    & 0 & 0 & 650  & 1 & 1 & 652  & 25.5     sec\\\hline
$f_2$ & 1362 & 1 & 1 & 0    & 0 & 0 & 1364 & 174.0    sec\\\hline
$f_3$ & 0    & 0 & 0 & 1843 & 1 & 1 & 1845 & 541.9    sec\\\hline
$f_4$ & 1877 & 1 & 1 & 0    & 0 & 0 & 1879 & 6919.2   sec\\\hline
$f_5$ & 0    & 0 & 0 & 1422 & 1 & 1 & 1424 & 84820.4  sec\\\hline
$f_6$ & 680  & 1 & 1 & 0    & 0 & 0 & 682  & 9144.2   sec\\\hline
$f_7$ & 86   & 1 & 1 & 0    & 0 & 0 & 88   & 0.3      sec\\\hline
\end{tabular}
\end{table}

\section{Parabolic invariants}
\label{sec:parab-invariant}

In this section, we study the parabolic invariants of the coinvariant
rings of finite Coxeter groups,
and we give the necessary and sufficient condition 
(Theorems~\ref{thm:necessary-condition-for-parabolic-invariant}
and \ref{thm:sufficient-condition-for-parabolic-invariant})
for the strong Lefschetz elements of the parabolic invariants
of Weyl groups.
Thus our second main theorem
(Theorem~\ref{thm:condition-for-partial-flag})
is proved by 
Theorems~\ref{thm:necessary-condition-for-parabolic-invariant}
and \ref{thm:sufficient-condition-for-parabolic-invariant}.

Let $W$ be a finite Coxeter group,
$S_0 \subset W$ the set of the simple reflections,
and $\Delta$ the corresponding root system.
For a subset $S \subset S_0$,
the subgroup $W_S$ of $W$ generated by $S$
is called a {\em parabolic subgroup} of $W$.
The {\em parabolic invariant} $R^{W_S}$  of $R$ is 
the graded subalgebra of the coinvariant ring $R$ of $W$,
defined by
\begin{align}
  \label{eq:def-of-parabolic-invariant}
  R^{W_S} &= \{ f \in R \;;\; \text{ $f$ is $W_S$-invariant} \}.
\end{align}
When $W$ is a Weyl group,
it is known that $R^{W_S}$ is isomorphic to the cohomology ring 
$H^\ast(G/P, \RR)$ of the partial flag variety $G/P$
with $R^{W_S}_d \simeq H^{2d}(G/P, \RR)$,
where $G$ is the connected and simply-connected complex Lie group
corresponding to the root system $\Delta$,
and $P$ is its parabolic subgroup whose Weyl group is equal to $W_S$
(see \cite[Theorem 5.5]{MR0429933} e.g.).
The following proposition describes the basic structure 
of the parabolic invariant $R^{W_S}$
\cite[Corollary 5.4]{MR630960}.

\begin{proposition}
\label{prop:basic-structure-of-parabolic-invariant}
Let $W^S$ be the set of the coset representatives of minimum length
of the coset space $W/W_S$.
We have the decomposition of the parabolic invariant $R^{W_S}$ 
into homogeneous components:
\begin{align*}
  R^{W_S} &= \bigoplus_{d=0}^{m-m_S} R^{W_S}_d,
  &
  \dim_{\RR} R^{W_S}_d = \#\{ w \in W^S \;;\; l(w) = d \},    
\end{align*}
where $m$ and $m_S$ are the lengths of the longest elements of 
$W$ and $W_S$, respectively.
\qed
\end{proposition}

The following theorem gives the necessary condition 
for strong Lefschetz elements of the parabolic invariant $R^{W_S}$,
and the theorem is the analogue of 
Theorem \ref{thm:necessary-condition}
which deals with the coinvariant rings $R$.

\begin{theorem}
\label{thm:necessary-condition-for-parabolic-invariant}
Let $R^{W_S} = \bigoplus_{d=0}^{m-m_S} R^{W_S}_d$ 
be the parabolic invariant of the coinvariant ring $R$ of 
a (possibly non-crystallographic) finite Coxeter group $W$,
defined in Equation $(\ref{eq:def-of-parabolic-invariant})$.
If $\ell \in R^{W_S}_1$ is a strong Lefschetz element of $R^{W_S}$,
then $\ell$ is not fixed by any reflections of $W \smallsetminus W_S$.
\end{theorem}
\begin{proof}
Let $\ell \in R^{W_S}_1$ be a strong Lefschetz element of $R^{W_S}$,
and $s \in W$ a reflection not contained in $W_S$.
Assume that $\ell$ is fixed by the action of $s$.
Since the multiplication map
$\times \ell^{m-m_S}: R^{W_S}_0 \to R^{W_S}_{m-m_S}$
is bijective,
$\ell^{m-m_S} \in R^{W_S}_{m-m_S}$ is nonzero.
It follows from the assumption that the nonzero element 
$\ell^{m-m_S} \in R^{W_S}_{m-m_S}$ is fixed by the
subgroup $W'$  of $W$ generated by $W_S$ and the reflection $s$.
Note that the number of reflections in $W'$ is greater than
that in $W_S$.

Fix a basis $X_{w_0}$ of the one-dimensional space $R_m$.
Define a symmetric bilinear form on $R$ by
$\langle f,g \rangle = a$, where $a$ is defined by
$f g = a X_{w_0} + (\text{lower degree terms})$.
In fact, this bilinear form gives a perfect pairing of 
$R_d$ with $R_{m-d}$.
Since $R_m$ consists of anti-invariants,
this bilinear form gives correspondence between
$W$-submodules $U$ of $R_d$ and $W$-submodules $U'$ of $R_{m-d}$,
where $U'$ is isomorphic to 
$(\text{sign representation}) \otimes_{\RR} (\text{dual of $U$})$.
Thanks to this correspondence,
there exists a nonzero anti-invariant 
with respect to the action of $W'$ in $R_{m_S}$.

In general, for a reflection group which has $m'$ reflections,
the degrees of anti-invariants are at least $m'$,
since anti-invariants should be divisible by any linear forms
which define reflecting hyperplanes.
As noted above,
the number of reflections in $W'$ is greater than $m_S$,
and this is a contradiction.
We thus have proved that
the strong Lefschetz element $\ell$ is not fixed by 
the reflection $s$.
\end{proof}

Next we prove that the condition in
Theorem~\ref{thm:necessary-condition-for-parabolic-invariant}
is also the sufficient condition, when $W$ is a Weyl group.
To prove that $\ell \in R^{W_S}_1$ is a strong Lefschetz element,
it is enough to show that
$w(\ell)$ is a strong Lefschetz element of $w(R^{W_S})$,
which is a graded subalgebra of $R$,
for an element $w \in W$. 
Since $w(R^{W_S}) = R^{w W_S w^{-1}}$,
it suffices to show that
\begin{enumerate}
  \item
  there exists a subset $S'$ of $S_0$
  such that $w W_S w^{-1} = W_{S'}$,
  \item
  $w(\ell) \in R^{W_{S'}}_1$
  is in the K\"ahler cone,
  that is,
  the inner product $(w(\ell), \alpha)$ is positive
  for any simple roots $\alpha$ corresponding to $S_0 \smallsetminus S'$,
\end{enumerate}
where $\Delta$ is realized in $R_1$,
and the $W$-invariant inner product $(,)$ is 
defined on $R_1$.
We prove the following theorem along this line.

\begin{theorem}
\label{thm:sufficient-condition-for-parabolic-invariant}
Let $R$ be the coinvariant ring of a Weyl group $W$.
Let $S$ be a subset of the set of the simple reflections of $W$.
Let $W_S$ be the subgroup of $W$ generated by $S$,
and $R^{W_S}$ the parabolic invariant.
If $\ell \in R^{W_S}_1$ is not fixed by any reflections in
$W \smallsetminus W_S$,
then $\ell$ is a strong Lefschetz element of $R^{W_S}$.
\end{theorem}
\begin{proof}
First we make some preparations on root systems.
Let $\ell \in R_1$ be an element satisfying the following two
conditions:
\begin{enumerate}
\item[(i)] 
$\ell \in R^{W_S}_1$, that is,
$\ell$ is fixed by any reflections in $W_S$,
\item[(ii)]
$\ell$ is not fixed by any reflections in $W \smallsetminus W_S$.
\end{enumerate}
Let $\Delta_S$ be the root subsystem of $\Delta$
generated by the simple roots corresponding to $S$,
and $n$ and $k$ the ranks of $\Delta$ and $\Delta_S$, respectively.
Let $C$ be a closed Weyl chamber containing $\ell$,
and $F$ its face defined as the intersection 
with the orthogonal complement of $\langle \Delta_S \rangle$,
where $\langle \Delta_S \rangle$ denotes the real vector subspace
of $R_1$ spanned by $\Delta_S$.
Note that
$\codim_{\RR} F \ge \dim_{\RR} \langle \Delta_S \rangle = k$,
and that $\ell \in F$ by Condition (i).
The face $F$ is an intersection of 
walls (i.e. faces of codimension one)
of $C$ defined by roots in $\Delta_S$ by Condition (ii).
Since the number of such walls are at most $k$,
we have $\codim_{\RR} F = k$.
We also remark that Condition (ii) is rephrased as
that $\ell$ is in the interior of $F$.
Let $\{ \alpha_1, \alpha_2, \ldots, \alpha_k \} \subset \Delta_S$
be the roots defining the $k$ walls containing $F$,
where $\alpha_i$ is chosen as $(\alpha_i, C) \ge 0$ for every $i$.
Then $\{ \alpha_1, \alpha_2, \ldots, \alpha_k \}$ is a basis of
the real vector space $\langle \Delta_S \rangle$,
and the subsystem $\langle \Delta_S \rangle \cap \Delta$
of $\Delta$ is generated by this basis,
since this basis becomes a subset of the simple system
when we regard $C$ as the fundamental Weyl chamber.
Note that $\langle \Delta_S \rangle \cap \Delta = \Delta_S$,
since $\Delta_S$ is a parabolic root subsystem.
Thus $\{ \alpha_1, \alpha_2, \ldots, \alpha_k \}$ is a base
(i.e. can be a simple system) of $\Delta_S$.

Based on the above preparations,
the theorem is proved along the line indicated before.
There exists $w \in W$ such that $w(C) = C_0$,
where $C_0$ is the closed fundamental Weyl chamber.
Thus the root subsystem $w(\Delta_S)$ of $\Delta$ has a base
$S' :=
\{ w(\alpha_1), w(\alpha_2), \ldots, w(\alpha_k) \} \subset \Delta$.
This base is a subset of the simple system of $\Delta$
corresponding to $S_0$,
since each $w(\alpha_i)$ defines a wall of $C_0$,
and $(w(\alpha_i), C_0) \ge 0$.
Therefore $S'$ satisfies $w(\Delta_S) = \Delta_{S'}$,
and hence $w W_S w^{-1} = W_{S'}$.

The image $w(\ell)$ is in the interior of the face $w(F)$ 
of the fundamental Weyl chamber $C_0$.
Therefore $w(l)$ is fixed by the reflections corresponding to
the roots in $\Delta_{S'}$,
and is not fixed by those corresponding 
to $\Delta \smallsetminus \Delta_{S'}$.
This means that $w(\ell) \in w(R^{W_S}) = R^{W_{S'}}$
is in the K\"ahler cone of the corresponding homogeneous space $G/P$.
Hence we have proved the theorem.
\end{proof}

Finally Theorems~\ref{thm:necessary-condition-for-parabolic-invariant}
and \ref{thm:sufficient-condition-for-parabolic-invariant}
prove our second main theorem
(Theorem~\ref{thm:condition-for-partial-flag}).

\bibliographystyle{amsalpha}
\bibliography{math}

\end{document}